\documentclass[12pt,a4paper]{article}
\usepackage{mathrsfs}
\usepackage{indentfirst}
\usepackage{latexsym,amsmath,amssymb,amsfonts,amsthm,epsfig,graphicx,cite,psfrag}
\usepackage{makeidx}
\usepackage{color}
\include{PDF}
\newtheorem{theorem}{Theorem}
\newtheorem{lemma}[theorem]{Lemma}

\newtheorem{corollary}[theorem]{Corollary}

\newtheorem{observation}{Observation}

\begin{document}
\title{\Large \bf{Erd\H{o}s-Gallai-type results for colorful\\ monochromatic connectivity
of a graph}\footnote{Supported by NSFC No.11371205, ``973" program
No.2013CB834204, and PCSIRT.}}

\author{\small Qingqiong~Cai, Xueliang~Li, Di~Wu\\
\small Center for Combinatorics and LPMC-TJKLC\\
\small Nankai University, Tianjin 300071, China\\
\small cqqnjnu620@163.com; lxl@nankai.edu.cn; wudiol@mail.nankai.edu.cn}
\date{}
\maketitle
\begin{abstract}
A path in an edge-colored graph is called a \emph{monochromatic
path} if all the edges on the path are colored the same. An
edge-coloring of $G$ is a \emph{monochromatic connection coloring}
(MC-coloring, for short) if there is a monochromatic path joining
any two vertices in $G$. The \emph{monochromatic connection number},
denoted by $mc(G)$, is defined to be the maximum number of colors
used in an MC-coloring of a graph $G$. These concepts were
introduced by Caro and Yuster, and they got some nice results.
In this paper, we will study two kinds of Erd\H{o}s-Gallai-type
problems for $mc(G)$, and completely solve them.


\end{abstract}

\section{Introduction}

All graphs considered in this paper are simple, finite and
undirected. We follow the terminology and notation of Bondy and
Murty \cite{Bondy}. For a graph $G$, we use $V(G)$, $E(G)$, $n(G)$,
$m(G)$, $\Delta(G)$ and $\delta(G)$ to denote the vertex set, edge
set, number of vertices, number of edges, maximum degree and minimum
degree of $G$, respectively. For $D\subseteq V(G)$, let $|D|$ be the
number of vertices in $D$, and $G[D]$ be the subgraph of $G$ induced
by $D$.

Let $G$ be a nontrivial connected graph with an edge-coloring $f:
E(G)\rightarrow \{1,2,\ldots,\ell\},$ $\ell \in \mathbb{N}$, where
adjacent edges may be colored the same. A path of $G$ is a
\emph{monochromatic path} if all the edges on the path are colored
the same. An edge-coloring of $G$ is a \emph{monochromatic
connection coloring} (MC-coloring, for short) if there is a
monochromatic path joining any two vertices in $G$. How colorful can
an MC-coloring be ? This question is the natural opposite of the
recently well-studied problem on rainbow connection number
\cite{CLR, Chartrand, KY, LSS, LiSun} for which we seek to find an
edge-coloring with minimum number of colors so that there is a
rainbow path joining any two vertices.

The \emph{monochromatic connection number} of $G$, denoted by
$mc(G)$, is defined to be the maximum number of colors used in an
MC-coloring of a graph $G$. An MC-coloring of $G$ is called
\emph{extremal} if it uses $mc(G)$ colors. An important property of
an extremal MC-coloring is that the subgraph induced by edges with
one same color forms a tree \cite{Caro}. For a color $i$, the
\emph{color tree} $T_i$ is the tree consisting of all the edges of
$G$ with color $i$. A color $i$ is \emph{nontrivial} if $T_i$ has at
least two edges; otherwise, $i$ is \emph{trivial}. A nontrivial
color tree with $t$ edges is said to \emph{waste} $t-1$ colors.
Every connected graph $G$ has an extremal MC-coloring such that for
any two nontrivial colors $i$ and $j$, the corresponding trees $T_i$
and $T_j$ intersect in at most one vertex \cite{Caro}. Such an
extremal coloring is called \emph{simple}.

These concepts were introduced by Caro and Yuster in \cite{Caro}. A
straightforward lower bound for $mc(G)$ is $m(G)-n(G)+2$. Simply
color the edges of a spanning tree with one color, and each of the
remaining edges may be assigned a distinct fresh color. Caro and
Yuster gave some sufficient conditions for graphs attaining this
lower bound.
\begin{theorem}[\cite{Caro}]\label{thm1}
Let $G$ be a connected graph with $n>3$. If $G$ satisfies any of the
following properties, then $mc(G)=m-n+2$.\\
$(a)$ $\overline{G}$ (the complement of $G$) is 4-connected.\\
$(b)$ $G$ is triangle-free.\\
$(c)$ $\Delta(G)<n-\frac{2m-3(n-1)}{n-3}$.
 In particular, this holds if $\Delta(G)\leq(n+1)/2$, and this also holds if $\Delta(G)\leq n-2m/n$.\\
$(d)$ $Diam(G)\geq3$.\\
$(e)$ $G$ has a cut vertex.
\end{theorem}

Moreover, the authors proved some nontrivial upper bounds for
$mc(G)$ in terms of the chromatic number, the connectivity and the
minimum degree. Recall that a graph is called
\emph{$s$-perfectly-connected} if it can be partitioned into $s+1$
parts $\{v\},V_1,\ldots,V_s$, such that each $V_j$ induces a
connected subgraph, any pair $V_j,V_r$ induces a corresponding
complete bipartite graph, and $v$ has precisely one neighbor in each
$V_j$. Notice that such a graph has minimum degree $s$, and $v$ has
degree $s$.

\begin{theorem}[\cite{Caro}]\label{thm2}
$(1)$ Any connected graph $G$ satisfies $mc(G)\leq m-n+\chi(G)$.\\
$(2)$ If $G$ is not $k$-connected, then $mc(G)\leq m-n+k$. This is sharp for any $k$.\\
$(3)$ If $\delta(G)=s$, then $mc(G)\leq m-n+s$, unless $G$ is $s$-perfectly-connected, in which case $mc(G)=m-n+s+1$.
\end{theorem}

In this paper, we will study two kinds of Erd\H{o}s-Gallai-type
problems for $mc(G)$.

\noindent\textbf{Problem A}: \ Given two positive integers $n$ and
$k$ with $1\leq k\leq {n \choose 2}$, compute the minimum integer
$f(n,k)$ such that if $|E(G)|\geq f(n,k)$, then $mc(G)\geq k$.

\noindent\textbf{Problem B}: Given two positive integers $n$ and $k$
with $1\leq k\leq {n \choose 2}$, compute the maximum integer
$g(n,k)$ such that if $|E(G)|\leq g(n,k)$, then $mc(G)\leq k$.

It is worth mentioning that the two parameters $f(n,k)$ and $g(n,k)$
are equivalent to another two parameters. Let
$t(n,k)=\min\{|E(G)|:|V(G)|=n, mc(G)\geq k\}$ and
$s(n,k)=\max\{|E(G)|:|V(G)|=n, mc(G)\leq k\}$. It is easy to see
that $t(n,k)=g(n,k-1)+1$ and $s(n,k)=f(n,k+1)-1$. This paper is devoted to
determining the exact values of $f(n,k)$ and $g(n,k)$ for all integers
$n$ and $k$ with $1\leq k\leq {n \choose 2}$; see Theorem \ref{thm3}
and Theorem \ref{thm4}.

\section{Main results}

\subsection{The result for $f(n,k)$}

We first state several lemmas, which will be used to determine the
value of $f(n,k)$.
\begin{lemma}\label{lem1}
Let $H$ be a connected graph on $n$ vertices, and $G$ a connected spanning subgraph of $H$.
If $mc(H)=m(H)-n+2$, then $mc(G)=m(G)-n+2$.
\end{lemma}
\begin{proof}
It suffices to prove that $mc(G)\leq m(G)-n+2$. At first, color the
edges of $G$ with $mc(G)$ colors such that there is a monochromatic
path joining any two vertices. Then, give each edge in $E(H)-E(G)$ a
different fresh color. Hereto we get an MC-coloring of $H$ using
$mc(G)+m(H)-m(G)$ colors, which implies that $mc(G)+m(H)-m(G)\leq
mc(H)$. Therefore, $mc(G)\leq
mc(H)-m(H)+m(G)=(m(H)-n+2)-m(H)+m(G)=m(G)-n+2$.
\end{proof}

\begin{lemma}\label{lem2}
Let $n$ and $p$ be two integers with $0\leq p\leq {n-1 \choose 2}$.
Then every connected graph $G$ with $n$ vertices and $m={n \choose 2}-p$ edges satisfies
$mc(G)\geq {n \choose 2}-2p$.
\end{lemma}
\begin{proof}
Proving that $mc(G)\geq {n \choose 2}-2p$ amounts to finding an MC-coloring of $G$
which wastes at most $p$ colors.
We distinguish the following two cases.

Case 1: $n-2\leq p\leq {n-1 \choose 2}$.

By the lower bound, we have $mc(G)\geq m-n+2\geq m-p={n \choose 2}-2p$.

Case 2: $0\leq p\leq n-3$.

Now consider the graph $\widetilde{G}$, which is obtained from
$\overline{G}$ by deleting all the isolated vertices. If
$n(\widetilde{G})\leq p+1(\leq n-2)$, then we can find at least two
vertices $v_1$, $v_2$ of degree $n-1$ in $G$. Take a star $S$ with
$E(S)=\{v_1v:v\in \widetilde{G}\}$. We give all the edges in $S$ one
color, and every other edge a different fresh color. Obviously, it
is an MC-coloring of $G$ which wastes at most $p$ colors. If
$n(\widetilde{G})\geq p+2$, say $n(\widetilde{G})=p+t$ ($t\geq2$),
then $\widetilde{G}$ has at least $t$ components (since
$m(\widetilde{G})=p$). If $\widetilde{G}$ has exactly two components
$C_1$ and $C_2$, then $t=2$, $n(C_j)\geq2$, and all the missing
edges of $G$ lie in $C_j$ for $j\in\{1,2\}$. Take a double star $S'$
as follows: one vertex from $C_1$ is adjacent to all the vertices in
$C_2$, and one vertex from $C_2$ is adjacent to all the
vertices in $C_1$. Give all the edges in $S'$ one color, and every
other edge in $G$ a different fresh color. Then we obtain an
MC-coloring of $G$, which wastes $p$ colors (since $S'$ has exactly
$p+1$ edges). If $\widetilde{G}$ has $\ell\geq 3$ components
$C_1,C_2,\ldots,C_{\ell}$, then $\ell\geq t$, $n(C_j)\geq2$, and all
the missing edges of $G$ lie in $C_j$ for $j\in\{1,2,\ldots,\ell\}$.
One vertex from $C_j$ is adjacent to every vertex in $C_{j+1}$ by a
fresh color $i_j$ for $j\in\{1,2,\ldots,\ell\}$ (cyclically, that is
a vertex from $C_{\ell}$ which is adjacent to every vertex in $C_1$
by the color $i_{\ell}$). Each other edge in $G$ receives a
different fresh color. Obviously, it is an MC-coloring of $G$, and
the number of wasted colors is
$\sum_{j=1}^{\ell}(n(C_j)-1)=p+t-\ell\leq p$.
\end{proof}
As an immediate consequence, we obtain the following corollary.
\begin{corollary}\label{cor1}
Let $n,\ p,\ k$ be three integers with $0\leq p\leq {n
\choose 2}/2$ and $k={n \choose 2}-2p$. Then $f(n,k)\leq {n \choose
2}-p$.
\end{corollary}

\begin{lemma}[\cite{Caro}]\label{lem3}
If $G$ is a complete $r$-partite graph, then $mc(G)=m-n+r$.
\end{lemma}

Given two positive integers $n$ and $t$ with $3\leq t\leq n$, let
$G_n^t$ be the graph defined as follows: partition the vertex set of
the complete graph $K_n$ into $t$ vertex classes
$V_1,V_2,\ldots,V_t$, where $\left||V_j|-|V_r|\right|\leq 1$ for
$1\leq j\neq r\leq t$; select a vertex $v_j^*$ from $V_j$ ($1\leq
j\leq t$), and delete all the edges joining $v_j^*$ to another
vertex in $V_j$. The remaining edges in $V_j$ ($1\leq j\leq t$) are
called \emph{internal edges}. Clearly, $m(G_n^t)={n \choose 2}-n+t$.
Next we will show that $mc(G_n^t)={n \choose 2}-2n+2t$. The proof is
similar to that of Lemma \ref{lem3}. We begin with an easy
observation.

\begin{observation}\label{obs1}
Let $f$ be an extremal MC-coloring of a connected graph $G$. Then
every nontrivial color tree in $f$ contains at least one pair of
nonadjacent vertices.
\end{observation}
\begin{proof}
Suppose that $T_i$ is a nontrivial color tree, in which all the
pairs of vertices are adjacent in $G$. Then we can adjust the
coloring of $T_i$. Color one edge of $T_i$ with color $i$, and give
each other edge of $T_i$ a different fresh color. Obviously, the new
coloring is still an MC-coloring, but uses more colors than $f$, a
contradiction.
\end{proof}

\begin{lemma}\label{lem4}
$mc(G_n^t)={n \choose 2}-2n+2t$.
\end{lemma}
\begin{proof}
Since $G_n^t$ contains a spanning complete $t$-partite graph, it
follows from Lemma \ref{lem3} that $mc(G_n^t)\geq m(G_n^t)-n+t={n
\choose 2}-2n+2t$. To prove the other direction, we need the
following three claims.

\noindent \textbf{Claim 1}: In any simple extremal MC-coloring $f$
of $G_n^t$, each nontrivial color tree intersects exactly two vertex
classes.\\

Suppose that a nontrivial color tree $T_i$ intersects $s\geq3$
vertex classes, say $V_1,V_2,\ldots,V_s$. Let $P_j=V(T_i)\cap V_j$
and $|P_j|=p_j$ for $1\leq j\leq s$. Denote by $x$ the number of
internal edges in $G[\bigcup_{j=1}^{s}P_j]$ (the subgraph of $G_n^t$
induced by $\bigcup_{j=1}^{s}P_j$). Then $G[\bigcup_{j=1}^{s}P_j]$
has $\sum_{1\leq j<r\leq s}p_jp_r+x$ edges in total. Observe that
$T_i$ has $\sum_{j=1}^{s}p_j-1$ edges, and since the coloring $f$ is
simple, each other edge in $G[\bigcup_{j=1}^{s}P_j]$ forms a trivial
color tree. Thus we get that $G[\bigcup_{j=1}^{s}P_j]$ contains
$\sum_{1\leq j<r\leq s}p_jp_r-\sum_{j=1}^{s}p_j+x+2$ colors. Now we
adjust the coloring of $G[\bigcup_{j=1}^{s}P_j]$. One vertex from
$P_j$ is adjacent to every vertex in $P_{j+1}$ by a fresh color
$i_j$ for $j\in\{1,2,\ldots,s\}$ (cyclically, that is a vertex from
$P_{s}$ which is adjacent to every vertex in $P_1$ by the color
$i_{s}$). Each other edge in $G[\bigcup_{j=1}^{s}P_j]$ receives a
different fresh color. Obviously, the new coloring is still an
MC-coloring, but now it uses $\sum_{1\leq j<r\leq
s}p_jp_r-\sum_{j=1}^{s}p_j+x+s$, contradicting to the fact that $f$
is extremal. Suppose that a nontrivial color tree $T_i$ intersects
only one vertex class, say $V_1$. Clearly, $v_1^*\notin V(T_i)$,
that is $T_i$ contains no pairs of nonadjacent vertices, a
contradiction. Thus each nontrivial color tree intersects exactly
two vertex classes.

\noindent \textbf{Claim 2}: There exists a simple extremal
MC-coloring of $G_n^t$ such that each nontrivial color tree is a
star or a double star, which does not contain any internal edges.

Let $f$ be a simple extremal MC-coloring of $G_n^t$ and $T_i$ a
nontrivial color tree in $f$. By Claim 1, we may assume that $T_i$
intersects $V_1$ and $V_2$ with $1\leq p_1\leq p_2$. Since $f$ is
simple, any edge in $G[P_1\bigcup P_2]$ but not in $T_i$ must be a
trivial color tree. Thus $G[P_1\bigcup P_2]$ contains
$p_1p_2-p_1-p_2+x+2$ colors. We distinguish the following two cases
(the case $p_1=p_2=1$ is excluded, since then $T_i$ has two
vertices, contradicting to the fact that $T_i$ is nontrivial).

{\bf Case 1:} $p_1=1$ and $p_2\geq2$

If $T_i$ is the star which consists of all the edges connecting
$P_1$ and $P_2$, then we are done. Otherwise, we replace $T_i$ with
this star, and color each other edge in $G[P_1\bigcup P_2]$ with a
different fresh color. Clearly, this change maintains an MC-coloring
without affecting the total number of colors. In other words, the
new coloring is still a simple extremal MC-coloring. Moreover, now
the nontrivial color tree in $G[P_1\bigcup P_2]$ is a star
containing no internal edges.

{\bf Case 2:} $2\leq p_1\leq p_2$.

If $T_i$ is a double star which consists of all the edges connecting
a certain vertex from $P_1$ and $P_2$, and all the edges connecting
a certain vertex from $P_2$ and $P_1$, then we are done. Otherwise,
we replace $T_i$ with one double star as stated above, and color
each other edge in $G[P_1\bigcup P_2]$ with a different fresh color.
Clearly, this change maintains an MC-coloring without affecting the
total number of colors. In other words, the new coloring is still a
simple extremal MC-coloring. Moreover, now the nontrivial color tree
in $G[P_1\bigcup P_2]$ is a double star containing no internal
edges.
\begin{figure}[ht]
\begin{center}
\includegraphics[width=10cm]{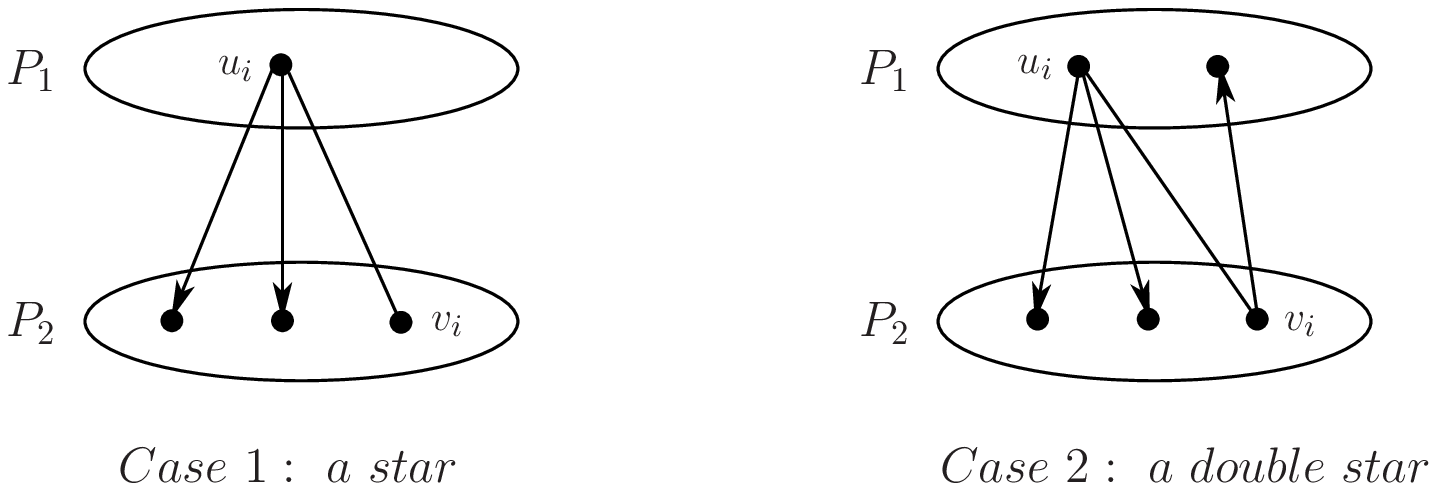}\\
Figure 1: The illustration of Claim 2.
\end{center}
\end{figure}

Now we may assume that every nontrivial color tree $T_i$ in $f$ is a
star or a double star containing no internal edges. In fact, the
stars can be viewed as degenerated double stars, by letting an
arbitrary leaf perform the role of the other center of a double
star. So we assume that all nontrivial color trees in $f$ are double
stars (some are possibly degenerated). For a nontrivial color tree
$T_i$, let $u_i$ and $v_i$ denote the two centers. Orient all the
edges of $T_i$ incident with $u_i$ other than $u_iv_i$ (if there are
any) as going from $u_i$ toward the leaves. Similarly, orient all
the edges of $T_i$ incident with $v_i$ other than $u_iv_i$ (if there
are any) as going from $v_i$ toward the leaves. Keep $u_iv_i$ as
unoriented. Since $T_i$ contains no internal edges, all of the edges
oriented from $u_i$ (if there are any) point to the same vertex
class (the vertex class of $v_i$), and all of the edges oriented
from $v_i$ (if there are any) point to the same vertex class (the
vertex class of $u_i$). Observe that the number of wasted colors of
$T_i$ is equal to the number of oriented edges in $T_i$.

\noindent \textbf{Claim 3}: For each $j$ ($1\leq j\leq t$), the
number of edges entering $V_j$ is at least $|V_j|-1$.

In order to solve the monochromatic connectedness of $|V_j|-1$ pairs
of nonadjacent vertices in $V_j$, there are double stars
$T_1,T_2,\ldots,T_{\ell}$ (some are possibly degenerated). Let $e_i
(1\leq i\leq \ell)$ denote the number of edges entering $V_j$ in
$T_i$. From Observation \ref{obs1}, it follows that $T_i (1\leq
i\leq \ell)$ must contain the vertex $v_j^*$. So $T_i (2\leq i\leq
\ell)$ covers at most $e_i$ vertices in $V_j$ but not in
$\bigcup_{q=1}^{i-1}T_q$. Thus we have
$(e_1+1)+\sum_{i=2}^{\ell}e_i\geq |V_j|$, that is,
$\sum_{i=1}^{\ell}e_i\geq |V_j|-1$.

Note that the total number of wasted colors in $f$ is equal to the
number of oriented edges in $G_n^t$. It follows from Claim 3 that
this number is at least $\sum_{j=1}^{t}(|V_j|-1)=n-t$. So we have
$mc(G)\leq ({n \choose 2}-n+t)-(n-t)={n \choose 2}-2n+2t$.
\end{proof}

We are now in the position to give the exact value of $f(n,k)$.
\begin{theorem}\label{thm3}
Given two positive integers $n$ and $k$ with $1\leq k\leq {n \choose 2}$,
\begin{displaymath}
f(n,k) = \left\{ \begin{array}{ll}
n+k-2 & \textrm{\ \ \ \ if\ $1\leq k\leq {n \choose 2}-2n+4$\ \ \ \ \ \ (1)}\\
{n \choose 2}+\left\lceil\frac{k-{n \choose 2}}{2}\right\rceil & \textrm{\ \ \ \
if\ ${n \choose 2}-2n+5\leq k\leq {n \choose 2}$\ \ \ (2)}
\end{array} \right.
\end{displaymath}
\end{theorem}

\begin{proof}
Let $G$ be a connected graph with $n$ vertices and $m$ edges.
Clearly, $f(n,1)=n-1$, so the assertion holds for $k=1$. If $2\leq
k\leq {n \choose 2}-2n+4$, by the lower bound we know that if $m\geq
n+k-2$, then $mc(G)\geq k$, which implies $f(n,k)\leq n+k-2$. To
prove $f(n,k)\geq n+k-2$, it suffices to find a connected graph $G$
satisfying $m=n+k-3$ and $mc(G)\leq k-1$. Let $H$ denote the graph
obtained from a copy of $K_{n-2}$ by adding two vertices $u$, $v$
and joining $u$ to some vertices in $K_{n-2}$ and joining $v$ to all
the other vertices in $K_{n-2}$. Obviously, $m(H)={n \choose 2}-n+1$
and $diam(H)=3$. Applying Theorem \ref{thm1}, we have $mc(H)={n
\choose 2}-2n+3$. In fact, $H$ is just the graph we want for $k={n
\choose 2}-2n+4$. For $2\leq k\leq {n \choose 2}-2n+3$, we take a
proper connected spanning subgraph $G$ of $H$ with $n+k-3$ edges. It
follows from Lemma \ref{lem1} that $mc(G)=k-1$. This completes the
proof of (1).

Proving (2) amounts to showing that if $k={n \choose 2}-2n+2t+1$ or
$k={n \choose 2}-2n+2t+2$ ($2\leq t\leq n-1$), then $f(n,k)={n
\choose 2}-n+t+1$. Let $k_1={n \choose 2}-2n+2t+1$, and $k_2={n
\choose 2}-2n+2t+2$. It follows from Corollary \ref{cor1} that
$f(n,k_2)\leq {n \choose 2}-n+t+1$. Since $f(n,k_1)\leq f(n,k_2)$,
if we prove $f(n,k_1)\geq {n \choose 2}-n+t+1$, then
$f(n,k_1)=f(n,k_2)={n \choose 2}-n+t+1$. So it suffices to find a
connected graph $G$ satisfying $m(G)={n \choose 2}-n+t$ and
$mc(G)\leq k_1-1={n \choose 2}-2n+2t$ for all $2\leq t\leq n-1$. If
$t=2$ (thus $n\geq 3$), then we can take $G=P_3,C_4$ for $n=3,4$,
respectively; for $n\geq 5$, we take the graph $G$ obtained from a
copy of $K_{n-2}$ by adding two adjacent vertices $u$, $v$ and
joining $u$ to exactly one vertex in $K_{n-2}$ and joining $v$ to
all the other vertices in $K_{n-2}$. It is easy to see that $m(G)={n
\choose 2}-n+2$, $\delta(G)=2$ and $u$ is the only vertex of degree
2. Since $G$ is not 2-perfectly-connected, it follows from Theorem
\ref{thm2} that $mc(G)\leq {n \choose 2}-2n+4$. If $3\leq t\leq
n-1$, then by Lemma \ref{lem4} we can take the graph $G_n^t$.
\end{proof}

\subsection{The result for $g(n,k)$}

We start with a useful lemma.
\begin{lemma}\label{lem5}
Let $G$ be a connected graph with $n$ vertices and $m$ edges. If
$\binom{n-t}{2}+t(n-t)\leq m\leq \binom{n-t}{2}+t(n-t)+(t-2)$ for
$2\leq t\leq n-1$, then $mc(G)\leq m-t+1$. Moreover, the bound is
sharp.
\end{lemma}

\begin{proof}
Let $f$ be a simple extremal MC-coloring of $G$. Suppose that $f$
contains $\ell$ nontrivial color trees $T_1,\ldots,T_{\ell}$, where
$t_i=|V(T_i)|$. Since $2\leq t\leq n-1$, we have $m\leq
\binom{n}{2}-1$, i.e., $G$ is not a complete graph. Thus $\ell\geq
1$. As $T_i$ has $t_i-1$ edges, it wastes $t_i-2$ colors. So it
suffices to prove that $\sum^{\ell}_{i=1}{(t_i-2)}\geq t-1$. Since
each $T_i$ can monochromatically connect at most $\binom{t_i-1}{2}$
pairs of nonadjacent vertices in $G$, we have
$$\sum^{\ell}_{i=1}{\binom{t_i-1}{2}}\geq \binom{n}{2}-m.$$
Assume that $\sum^{\ell}_{i=1}{(t_i-2)}<t-1$, namely,
$\sum^{\ell}_{i=1}{(t_i-1)}<t-1+\ell$. As each $T_i$ is nontrivial,
we have $t_i-1\geq 2$, thus $1\leq \ell\leq t-2$. By straightforward
convexity, the expression $\sum^{\ell}_{i=1}\binom{t_i-1}{2}$,
subject to $t_i-1\geq 2$, is maximized when $\ell-1$ of the $t_i's$
are equal to 3, and one of the $t_i's$, say $t_{\ell}$, is as large
as it can be, namely, $t_{\ell}-1$ is the largest integer smaller
than $t-1+\ell-2(\ell-1)=t-\ell+1$. Hence $t_{\ell}-1=t-{\ell}$. Now
\begin{align*}
\sum^{\ell}_{i=1}{\binom{t_i-1}{2}} & \leq (\ell-1)+\binom{t-\ell}{2}\\
& =\frac{1}{2}\left(t^2-t-2+\ell^2+(3-2t)\ell\right)\\
& \leq \binom{t-1}{2}\ \ (\text{take}~\ell=1)\\
& <\binom{t-1}{2}+1.
\end{align*}
For a contradiction, we just need to show that $\binom{t-1}{2}+1\leq
\binom{n}{2}-m$. In fact,
\begin{align*}
\binom{t-1}{2}+1+m & \leq \binom{t-1}{2}+1+\binom{n-t}{2}+t(n-t)+(t-2)\\
& =\binom{n}{2}.
\end{align*}

Next we will show that the bound is sharp. Let $G$ be the graph
defined as follows: at first, take a complete $(n-t+1)$-partite
graph $K$ with vertex classes $V_1,\ldots,V_{n-t+1}$ such that
$|V_j|=1$ for $1\leq j\leq n-t$, and $|V_{n-t+1}|=t$; then, add the
remaining edges (at most $t-2$) to $V_{n-t+1}$ randomly. Now assign
the edges between $V_1$ and $V_{n-t+1}$ with one color, and every
other edge a distinct fresh color. It is easily checked that this is
an $MC$-coloring of $G$ using $m-t+1$ colors, which implies
$mc(G)\geq m-t+1$. Hence $mc(G)=m-t+1$.
\end{proof}
With the aid of Lemma \ref{lem5}, we determine the exact value of
$g(n,k)$.

\begin{theorem}\label{thm4}
Given two positive integers $n$ and $k$ with $1\leq k\leq {n \choose 2}$,
\begin{align*}
g(n,k)=
\begin{cases}
\binom{n}{2} & \textrm{\ \ \ \ if $k=\binom{n}{2}$}\\
k+t-1 & \textrm{\ \ \ \ if $\binom{n-t}{2}+t(n-t-1)+1\leq k\leq \binom{n-t}{2}+t(n-t)-1$}\\
k+t-2 & \textrm{\ \ \ \ if $k=\binom{n-t}{2}+t(n-t)$}
\end{cases}
\end{align*}
for $2\leq t\leq n-1$.
\end{theorem}

\begin{proof}
If $k=\binom{n}{2}$, then clearly $g(n,k)=\binom{n}{2}$. If
$\binom{n-t}{2}+t(n-t-1)+1\leq k\leq \binom{n-t}{2}+t(n-t)-1$ for
$2\leq t\leq n-1$, it follows from Lemma \ref{lem5} that if $m(G)\leq
k+t-1$, then $mc(G)\leq k$. Hence, $g(n,k)\geq k+t-1$. Now let $G'$
be the graph as described in Lemma \ref{lem5} with $k+t$ edges.
Then $mc(G')=k+1>k$ for $\binom{n-t}{2}+t(n-t-1)+1\leq k\leq
\binom{n-t}{2}+t(n-t)-2$, and $mc(G')=k+2>k$ for
$k=\binom{n-t}{2}+t(n-t)-1$. So we have $g(n,k)\leq k+t-1$, and thus
$g(n,k)= k+t-1$. If $k=\binom{n-t}{2}+t(n-t)$ for $2\leq t\leq n-1$,
it follows from Lemma \ref{lem5} that if $m(G)\leq k+t-2$, then
$mc(G)\leq k-1<k$. Hence, $g(n,k)\geq k+t-2$. Now let $G''$ be the
graph as described in Lemma \ref{lem5} with $k+t-1$ edges. Then
$mc(G'')=k+1>k$. So we have $g(n,k)\leq k+t-2$, and thus $g(n,k)=
k+t-2$.
\end{proof}

\end{document}